\documentclass[12pt,centertags,oneside]{amsart}

\usepackage{graphicx}
\usepackage{amscd,amsxtra}
\usepackage{amsfonts}
\usepackage{amsthm}
\usepackage{amssymb,latexsym,amsmath}
\usepackage[backref=page]{hyperref}
\usepackage{cmmib57}
\usepackage{url}
\usepackage{color}
\usepackage{cancel}
\usepackage{braket}
\usepackage[stretch=10]{microtype}
\usepackage[normalem]{ulem}
\usepackage[all]{xy}

\usepackage[a4paper,width=16.2cm,top=3cm,bottom=3cm]{geometry}

\numberwithin{equation}{section}

\theoremstyle{plain}
\newtheorem{thm}{Theorem}[section]
\newtheorem{theorem}[thm]{Theorem}

\newtheorem{lemma}[thm]{Lemma}
\newtheorem{lem}[thm]{Lemma}

\newtheorem{cor}[thm]{Corollary}
\newtheorem{proposition}[thm]{Proposition}
\newtheorem{pro}[thm]{Proposition}
\theoremstyle{definition}
\newtheorem{remark}[thm]{Remark}

\newtheorem{claim}[thm]{Claim}

\newtheorem{s:examples}[thm]{s:examples}
\newtheorem{conjecture}[thm]{Conjecture}

\newtheorem{question}[thm]{Question}

\numberwithin{equation}{section}

\newcommand{\ga}[2]{\begin{gather}\label{#1}#2 \end{gather}}

\newcommand{\sO}{{\mathcal O}}

\newcommand{\sZ}{{\mathcal Z}}

\newcommand{\C}{{\mathbb C}}

\renewcommand{\P}{{\mathbb P}}
\newcommand{\Q}{{\mathbb Q}}
\newcommand{\R}{{\mathbb R}}

\newcommand{\Z}{{\mathbb Z}}

\newcommand{\oQ}{{\overline{\mathbb Q}}}
\newcommand{\obfx}{{{\rm Orb}_f(x)}}
\newcommand{\Aut}{{\rm Aut\hspace{.1ex}}}

\newcommand{\Bir}{{\rm Bir\hspace{.1ex}}}
\newcommand{\Rat}{{\rm Rat\hspace{.1ex}}}
\newcommand{\End}{{\rm End\hspace{.1ex}}}
\newcommand{\Al}{{\rm Alb\hspace{.1ex}}}
\newcommand{\alb}{{\rm alb}}

\newcommand{\ssec}{\subsection}

\newcommand{\ol}{\overline}
\newcommand{\ti}[1]{\tilde{#1}}
\newcommand{\ul}{\underline}
\newcommand{\vast}{\bBigg@{4}}
\newcommand{\Vast}{\bBigg@{5}}
\newcommand{\wt}{\widetilde}

\newcommand{\bk}{\mathbf{k}}

\newcommand{\cS}{\mathcal{S}}

\newcommand{\cZ}{\mathcal{Z}}

\newcommand{\Qbar}{\ol{\Q}}

\newcommand{\gS}{\Sigma}

\renewcommand{\ga}{\alpha}

\newcommand{\gd}{\delta}

\newcommand{\gk}{\kappa}
\newcommand{\gl}{\lambda}

\newcommand{\Alb}{\mathrm{Alb}}

\newcommand{\Id}{\mathrm{Id}}

\newcommand{\Orb}{\mathrm{Orb}}

\newcommand{\pr}{\mathrm{pr}}

\newcommand{\SL}{\mathrm{SL}}

\newcommand{\bss}{\backslash}

\newcommand{\cnec}{\mathrel{:=}}
\newcommand{\cto}{\circlearrowleft}

\newcommand{\dto}{\dashrightarrow}

\newcommand{\xto}[1]{\xrightarrow{ #1 }}
\newcommand{\xlto}[1]{\xleftarrow{ #1 }}

\title [irregular varieties]{On the Kawaguchi--Silverman Conjecture for birational automorphisms of irregular varieties}

\dedicatory{In memory of the late Professor Nessim Sibony}

\author{Jungkai Alfred Chen}
\address{Department of Mathematics, National Taiwan University,
	No. 1, Sec. 4, Roosevelt Rd., Taipei 10617, Taiwan.}
\email{jkchen@ntu.edu.tw}

\author{Hsueh-Yung Lin}
\address{Department of Mathematics, National Taiwan University,
	No. 1, Sec. 4, Roosevelt Rd., Taipei 10617, Taiwan.}
\email{hsuehyunglin@ntu.edu.tw}

\author{Keiji Oguiso}
\address{Mathematical Sciences, the University of Tokyo, Meguro Komaba 3-8-1, Tokyo, Japan, and National Center for Theoretical Sciences, Mathematics Division, National Taiwan University,
Taipei, Taiwan}
\email{oguiso@ms.u-tokyo.ac.jp}

\thanks{JAC is partially supported by NSTC 110-2123-M-002-005 and by NCTS.
	HYL is partially supported by
	Taiwan Ministry of Education Yushan Young Scholar Fellowship (NTU-110VV006),
	and NCTS (110-2628-M-002-006-).
	KO is supported by JSPS Grant-in-Aid (A) 15H05738, JSPS Grant-in-Aid (B) 15H03611, and by NCTS Scholar Program.}


\begin{document}

\begin{abstract} 
	We study the main open parts of the Kawaguchi--Silverman Conjecture, asserting that for a birational self-map $f$ of a smooth projective variety $X$ defined over $\oQ$, the arithmetic degree $\alpha_f(x)$ exists and coincides with the first dynamical degree $\delta_f$ for any $\Qbar$-point $x$ of $X$ with a Zariski dense orbit. Among other results, we show that this holds when $X$ has Kodaira dimension zero and irregularity $q(X) \ge \dim X -1$ or $X$ is an irregular threefold (modulo one possible exception). We also study the existence of Zariski dense orbits, with explicit examples.
\end{abstract}

\maketitle

\section{Introduction}

\ssec{Kawaguchi--Silverman Conjecture}
\hfill

Let $X$ be a smooth complex projective variety of dimension $d$, and let $f : X \dto X$ be a dominant rational self-map.
The \emph{first dynamical degree} of $f$, defined by the limit
\begin{equation}\label{eqn-defdd}
	\gd_f \cnec \lim_{n \to \infty} ((f^n)^*H \cdot H^{d-1})^{\frac{1}{n}},
\end{equation}
where $H$ is an ample divisor,
is a real number $\gd_f \ge 1$ which measures the complexity of $f$. We refer to \S\ref{notation} for the definition of the pullback of a divisor under a dominant rational map.
The existence of this limit is proved in the seminal work of Dinh and Sibony~\cite{DS05},
and is independent of the choice of $H$ (see also~\cite[Theorem 1.1]{Tr15} and~\cite[Theorem 1.1]{Tr20}).

Now, assume that the variety $X$ and the map $f : X \dto X$ are both defined over $\Qbar$.
We then have another way of measuring the complexity of $f$,
by considering the growth of the Weil heights of the iterations of $\Qbar$-points under $f$.
Such an arithmetic approach was
first conceived and studied by Silverman (\cite{Si91}, see also~\cite{CS93}),
but many fundamental questions,
including the existence of the so-called \emph{arithmetic dynamical degree},
still remain open.
A series of conjectural statements in this regard were formulated by
Kawaguchi and Silverman~\cite[Conjecture 6]{KS16b}, which we now recall.

Throughout the paper,
$\Rat(X)$ denotes the semi-group of rational dominant maps $f : X \dto X$ over its field of definition,
and $\Bir(X) \subset \Rat(X)$ the subgroup of birational maps.
For every $f \in \Rat(X)$,
let $U(f)$ be the largest Zariski open subset of $X$ such that
$f_{|U(f)}$ is a morphism,
and define
$$X_f(\oQ) := \{x \in X(\oQ)\,|\, f^n(x) \in U(f)\, \text{ for all } n \in \Z_{\ge 0}\}.$$
Let $H$ be an ample divisor on $X$ and let $h_H \ge 1$
be a height function associated to $H$ (see e.g. \cite[Part B, Theorem B.3.6]{HS00}).

\begin{conjecture}[Kawaguchi--Silverman Conjecture (KSC)]\label{conj1} Let $X$ be a smooth projective variety defined over $\Qbar$ and
	let $f \in \Rat(X)$.
	Then:
	\begin{enumerate}
		\item The limit
		$$a_f(x) := \lim_{n \to \infty} h_H(f^n(x))^{\frac{1}{n}}$$
		exists for any $x \in X_f(\oQ)$ and the value is an algebraic integer.
		
		\item For any $x \in X_f(\oQ)$ whose orbit
		$\obfx := \{f^n(x)\,|\, n \in \Z_{\ge 0}\}$
		under $f$ is Zariski dense in $X$, we have
		$$a_f(x) = \delta_f.$$
		
	\end{enumerate}

\end{conjecture}

\begin{remark}
		It is worth noticing that $\delta_f$ is always an algebraic integer
		when $f$ is a morphism (see e.g. \cite[Remark 35]{KS16b}),
		but there exist examples of transcendental
		$\delta_f$ when $f$ is not a morphism,
		by recent striking results~\cite{BDJ20} and \cite{BDJK21}.
		Therefore Conjecture~\ref{conj1} needs to be corrected,
		and very likely $a_f(x)$ could be a transcendental number as well.
	\end{remark}

	When the limit $a_f(x)$ in KSC (1) exists,
the value $a_f(x)$ is called the {\it arithmetic degree} of $f$ at $x$.
While KSC (1) is unknown,
we can always define
$$\overline{a}_f(x) := \limsup_{n \to \infty} h_H(f^n(x))^{\frac{1}{n}} \ge \underline{a}_f(x) := \liminf_{n \to \infty} h_H(f^n(x))^{\frac{1}{n}} \ge 1;$$
these are called the {\it upper arithmetic degree} and {\it lower arithmetic degree} of $f$ at
 $x \in X_f(\oQ)$ respectively.
It is known that $\overline{a}_f(x)$ and $\underline{a}_f(x)$
do not depend on the choice of an ample divisor $H$~\cite[Proposition 12]{KS16b}.

KSC (1) is affirmative when $f$ is a morphism~\cite[Theorem 3]{KS16a},
but very little is known in general, even when $f$ is birational.
KSC (2) is open even when $f$ is a morphism.
The only general result we know so far is that
	\begin{equation}\label{ineq-dalimsup}
		\delta_f \ge \overline{a}_f(x) 
	\end{equation}
	for any $f \in \Rat(X)$ and $x \in X_f(\oQ)$,
	due to Kawaguchi and Silverman~\cite[Theorem 4]{KS16a}
	and Matsuzawa~\cite[Theorem 1.4]{Ma20}.
	We also know that
	a $\Qbar$-point $x \in  X(\oQ)$ such that $a_f(x) = \delta_f$ exists
	when $f : X \to X$ is a surjective morphism,
	due to Matsuzawa, Sano and Shibata~\cite[Corollary 9.3]{MSS18}. 
	More remarkably, they prove that
	such $\Qbar$-points $x$ are Zariski dense in $X$.
	See e.g. \cite[Introduction]{LS21} for the current status of KSC (2).

\ssec{The existence of Zariski dense orbits}
\hfill

The statement KSC (2) is meaningful only when
there exists $x \in X_f(\oQ)$ such that $\obfx$ is Zariski dense in $X$,
otherwise we say that \emph{KSC (2) vacuously holds}.
The existence of such $x \in X_f(\oQ)$
seems to be less studied in the context of KSC.

\begin{question}\label{quest1} Let $X$ be a smooth projective variety over $\Qbar$
	and let $f \in \Rat(X)$ be a self-map of infinite order. 
	When does $f$ have Zariski dense $\Qbar$-orbits?
		Namely, when does
			$$\cZ(f)\cnec \Set{x \in X_f(\oQ) | \obfx \text{  is Zariski dense in } X} \ne \emptyset?$$
\end{question}

In Question~\ref{quest1}, we have an affirmative answer
if we replace the Zariski density of $\obfx$ by the weaker property that
$\obfx$ is infinite:
in~\cite{Am11},
Amerik even showed that the subset
$$\Set{x \in X_f(\oQ) | {|\obfx|} = \infty }$$
(and in particular $X_f(\oQ)$)
is Zariski dense in $X$.
When $X$ is a smooth projective surface and $f$ is birational with $\delta_f >1$,
the existence of Zariski dense orbit  is proved by Xie~\cite[Theorem 1.3]{Xi15} (See also \cite{Xi25}).

Directly related to Question~\ref{quest1}
is the so-called Zariski Dense Orbit conjecture,
which would provide 
a geometric characterization of $f \in \Rat(X)$ having Zariski dense $\Qbar$-orbits.

\begin{conjecture}[See e.g.{\cite[Conjecture 1.1]{JSXZ21}}]\label{ZDOconj} 
	Let $X$ be a smooth projective variety over $\Qbar$ and let $f \in \Rat(X)$ be a self-map of infinite order. Then $\cZ(f) \ne \emptyset$ 
	as long as there is no non-constant rational function
	$g$ on $X$ such that $g \circ f = g$. 
\end{conjecture}

Conjecture~\ref{ZDOconj} is known to hold for abelian varieties~\cite{GS17}. 
We will not say much about Conjecture~\ref{ZDOconj} in this work (apart from Remark~\ref{rem-ZODqd-1});
for a more general formulation, history and the current status of Conjecture~\ref{ZDOconj}
see e.g.~\cite[Introduction]{JSXZ21}.

In the negative direction,
a result due to Nakayama and Zhang (see Theorem~\ref{thmNZ}) implies that
when $\kappa (X) \ge 1$, a dominant rational self-map $f: X \dto X$ never has Zariski dense orbits in $X$.

\ssec{Main results}
\hfill

The aim of this paper is to study Conjecture~\ref{conj1} (2) and Question~\ref{quest1} for
{\it irregular} smooth projective varieties $X$ over $\Qbar$,
namely $X$ such that $q(X) := \dim H^1(X, \sO_X) > 0$.
We will prove the following three theorems.

\begin{theorem}\label{mainthm2}
	Let $X$ be a smooth projective variety with Kodaira dimension $\gk(X) = 0$ and irregularity $q(X) \ge \dim X -1$. Then
	KSC (2) holds for $f \in \Bir (X)$.
More precisely,
	
	\begin{enumerate}
		\item If $q(X) = \dim X$, then KSC (2) holds for all $f \in \Rat (X)$.
		\item If $q(X) = \dim X -1$, then
		KSC (2) vacuously holds for all $f \in \Bir (X)$.
	Namely, $\cZ(f)$ is always empty.
	\end{enumerate}
\end{theorem}

For the statement, we recall that if $\kappa (X) = 0$, then $q (X) \le \dim X$ and the equality holds exactly when $X$ is birational to an abelian variety by a fundamental result of Kawamata~\cite[Theorem 1]{Ka81}. Note that KSC (2) holds for $g \in \End_{{\rm var}}(A)$ of an abelian variety $A$ (\cite[Theorem 4]{KS16a} and \cite[Theorem 1.2]{Si17}).

\begin{theorem}\label{mainthm1}
	Let $X$ be a smooth projective variety and $f \in \Bir (X)$.
	Assume that $\dim\, (X) = 3$ and $q(X) > 0$. Then:
	
	\begin{enumerate}
		
		\item KSC (2) 
		is affirmative for $(X, f)$ unless $X$ is covered by rational surfaces.
		
		\item KSC (2)
		is affirmative for $(X, f)$ if
		$f \in \Aut (X)$.
	\end{enumerate}
\end{theorem}

Note that KSC (2) is unknown for rational surfaces\footnote{This is now known by \cite{Xi24}.} including $\P^2$; see~\cite[Theorem 1.4]{MSS18} for the known result in dimension $2$.
This is related to the exception in Theorem~\ref{mainthm1} (1). 

Through our attempt of studying the Kawaguchi--Silverman conjecture, 
we feel that one of the main 
difficulties in generalizing results 
from $\End (X)$ to $\Rat (X)$ (or even just $\Bir(X)$) is the lack of 
a good functorial property of the height functions under rational self-maps (see e.g. \cite[Exercise B2]{HS00} for a concrete example).  
At various places we need to circumvent the lack of such functoriality by means of 
algebro-geometric arguments, which is possible in the setting of
Theorems~\ref{mainthm2} and~\ref{mainthm1}.

The last result concerns Question~\ref{quest1}.
Let $f \in \Bir(X)$ where $X$ is a smooth projective variety over $\Qbar$ of dimension $d \ge 2$.
We have mentioned that $\cZ(f) \ne \emptyset$ only if $\gk(X) \le 0$.
Another necessary condition of $\cZ(f) \ne \emptyset$
is that the Albanese map $\alb : X \to \Alb(X)$ is surjective
(see Proposition~\ref{pro-albsurj}).
So $\cZ(f) \ne \emptyset$ implies that $(\gk(X),q(X))$
is in the following set:
\begin{equation}
\begin{split}
\cS_{d} & \cnec \Set{(\gk(X),q(X)) | \dim X = d,  \gk(X) \le 0,  \alb : X \to \Alb(X) \text{ surjective}} \\
& = \Set{(0,0), \ldots, (0, d), (-\infty,0), \ldots, (-\infty, d-1) }.
\end{split}
\end{equation}
In Theorem~\ref{mainthm2}.(2),
we have seen that a birational map $f \in \Bir(X)$ never has
Zariski dense $\Qbar$-orbits when
$(\gk(X),q(X)) = (0,d-1)$.
Our last result shows that this is the only exception among the list $\cS_{d}$,
by constructing explicit examples realizing each remaining pair.

\begin{theorem}\label{mainthm3}
	Let $X$ be a smooth projective variety and $f \in \Bir (X)$.
	Assume that $d:= \dim\, (X) \ge 2$.
 Then:
	
	\begin{enumerate}
		\item
		$f$ has Zariski dense $\Qbar$-orbits only if
		$$(\kappa (X), q(X)) \in \cS_{d} \setminus \{(0,d-1)\}.$$
		\item Conversely, for each $(\kappa, q) \in \cS_{d}$ such that $(\kappa, q) \ne (0,d-1)$,
		there exist $X$ and $f \in \Aut (X)$
		with $(\kappa (X), q(X)) = (\kappa, q)$ such that 
		$f$ has Zariski dense $\Qbar$-orbits. 
Moreover, there exists such $f$ with $\delta_f > 1$ 
			except when $(d,\gk,q) = (2,-\infty,1)$.
	\end{enumerate}
\end{theorem}

Our paper is organized as follows.
\S\ref{notation} and \S\ref{sect2} are preliminary sections; 
in \S\ref{notation}, we introduce some notations with brief explanations and in
\S\ref{sect2}, we prove several elementary lemmas, some of which are 
well-known to experts. We prove Theorems~\ref{mainthm2}, \ref{mainthm1} and \ref{mainthm3} in \S\ref{irranydim}, \S\ref{sect3} and \S\ref{sect4} respectively.

\subsection*{Acknowledgements}
\hfill

We would like to express our thanks to the NCTS and the staff members there for financial support and hospitality. This joint work was initiated during the third named author's stay at the NCTS in October 7 - December 16 2021.
We thank Professors Yohsuke Matsuzawa and Junyi Xie for helpful comments,
as well as Professors Shu Kawaguchi, De-Qi Zhang and Doctor Long Wang for their interest in this work and discussions. We also would like to express our thanks for the referee for her/his careful reading and valuable comments.

\section{Notation}\label{notation}

In this paper, we fix an algebraic closure $\oQ$ of the field of rational numbers $\Q$.
Unless stated otherwise, all varieties and rational maps among them are defined over $\oQ$.
For an abelian group $A$ and a field $K$, we denote by $A_K$ the $K$-vector space $A \otimes_{\Z} K$.

Let $X$ be a normal projective variety defined over $\oQ$.
For any Cartier divisor $D$, we denote by
$$h_D : X(\oQ) \to \R$$
the {\it logarithmic Weil height function} associated to $D$.
The function $h_D$ is unique up to bounded functions on $X(\oQ)$.
When $H$ is an ample divisor, we may and will choose $h_H$ so that $h_H \ge 1$ on $X(\oQ)$.
The reader is referred to~\cite[Part B, Theorem B.3.6]{HS00} for the basic properties on $h_D$ and the Weil height machinery.

Let $X$ be a normal projective variety over a field $\bk$.
Recall the definition of $\Rat (X)$ and $\Bir(X)$ from the introduction.
The subsets of (regular) endomorphisms and automorphisms of $X$ are denoted by
$\End (X)\subset \Rat (X)$ and $\Aut (X) \subset \Bir (X)$  respectively.

Let $f : X \dasharrow Y$ be a dominant rational map to a projective variety $Y$. We set
$$I(f) := X \setminus U(f),$$
where $U(f)$ is the largest Zariski open subset of $X$ on which $f$ is defined as a morphism to $Y$. We call $I(f)$ the indeterminacy set of $f$. Since $X$ is normal and $Y$ is projective, $I(f)$ is of codimension at least $2$ in $X$.

Assume that $X$ is smooth. Let
\begin{equation}\label{res-normal}
	X \xlto{p} \wt{X} \xto{q} Y
\end{equation}
be the resolution of $f$ by normalizing the graph of $f$.
For any Cartier divisor $D$ on $Y$,  define
$$f^*D \cnec p_*q^*D$$
which is a Cartier divisor.
It is easy to see that if we replace $\wt{X}$ in~\eqref{res-normal} by any other normal projective variety birational to $X$,
then $f^*D$ still defines the same Cartier divisor.

The Kodaira dimension of a smooth projective variety $X$ is denoted by $\kappa (X)$.
It is a birational invariant and when $X$ is not smooth,
$\kappa (X)$ always means the Kodaira dimension of any projective resolution of $X$.

\section{Some general lemmas}\label{sect2}

In this section, we prove some elementary lemmas 
which we will use in the rest of this article.

\ssec{First dynamical degree}
\hfill

\begin{lemma}\label{lem32}
	Let $\varphi : X \dasharrow B$ be a dominant rational map from a smooth projective variety $X$ to a smooth projective variety $B$ such that $\dim B = \dim X$ or $\dim B = \dim X -1$. Let $f \in \Bir (X)$. Assume that there exists $f_B \in \Rat (B)$ such that $f_B \circ \varphi = \varphi \circ f$.
	Then necessarily $f_B \in \Bir (B)$ and $\delta_f = \delta_{f_B}$.
\end{lemma}

\begin{proof} The first assertion is clear as $f \in \Bir (X)$.
	As the first dynamical degree is a birational invariant~\cite[Corollaire 7]{DS05}, we may assume that $\varphi$ is a surjective morphism by taking a projective resolution of the rational map $f$.
	By the product formula for the relative dynamical degree (\cite{DNT12}, \cite[Theorem 1.4]{Tr15}), we have
	$$\delta_f = d_1(f) = \max (d_1(f|\varphi)d_0(f_B), d_0(f|\varphi)d_1(f)).$$
	Here $d_k(f|\varphi)$ denotes the relative $k$-th dynamical degree of $f$ (\cite{DNT12}, \cite[Theorem 1.4]{Tr15}).
	Since $f$ is birational and $\dim B = \dim X$ or $\dim X -1$, it follows from the definition of $d_k(f|\varphi)$ that
	$$d_1(f|\varphi) = d_0(f|\varphi) = 1.$$
	Since $d_1(f_B) \ge 1 = d_0(f_B)$ (for any $f_B \in \Rat (B)$), it follows that
	$$\delta_f = \max (d_0(f_B), d_1(f_B)) = d_1(f_B) = \delta_{f_B}.$$
\end{proof}

\ssec{Zariski density of orbits}
\hfill

Recall that
$$\cZ(f)\cnec \Set{x \in X_f(\oQ) | \obfx \text{  is Zariski dense in } X}$$
for every rational dominant self-map $f:X \dto X$ of  a smooth projective variety
$X$ over $\Qbar$.

\begin{lem}\label{lem-itZ}
	Let $X$ be a smooth projective variety and let $f \in \Rat(X)$. For any $N \in \Z_{>0}$,
	we have
	$$\cZ(f) = \cZ(f^N) \cap X_f(\oQ)$$
\end{lem}
\begin{proof}
This follows from
$$\Orb_f(x) = \cup_{k=0}^{N-1} f^k(\Orb_{f^N}(x))$$
for any $x \in X_f(\oQ)$.
\end{proof}

\begin{lemma}\label{lem-dorb}
	Let $\pi : X \dto Y$ be a dominant rational map between irreducible varieties.
	Let $f \in \Rat (X)$ and $g \in \Rat (Y)$ such that 
$$\pi \circ f = g \circ \pi.$$
	Let $x \in \cZ(f)$.
	Then
	$$\Set{\pi \circ f^k(x) \in Y | k \in \Z_{\ge 0} \text{ and }  f^k(x) \notin I(\pi)}$$
	is dense in $Y$.
\end{lemma}

\begin{proof}
	Assume to the contrary that
	$\Set{\pi \circ f^k(x) \in Y |  f^k(x) \notin I(\pi)}$
	is contained
	in a proper Zariski closed subset $Z \subsetneq Y$.
	Since $\pi \circ f = g \circ \pi$, we would thus have
	$$\obfx \subset \{ \pi^{-1}(Z) \cup I(\pi) \} \subsetneq X,$$
	where $\pi^{-1}(Z)$ denotes the closure of
	$\pi^{-1}(Z \cap \pi(X \bss I(\pi)))$ in $X$, contradicting the assumption that $x \in \cZ(f)$.
\end{proof}

The following result due to Nakayama--Zhang for $f \in \Rat (X)$ (see \cite[Theorem A]{NZ09})
is fundamental, 
which is a generalization of the result of Nakamura-Ueno-Deligne for $f \in \Bir (X)$ (see \cite[Theorem 14.10]{Ue75}):

\begin{theorem}\label{thmNZ}
	Let $X$ be a smooth projective variety and $f \in \Rat (X)$.
	Let
	$$\phi_{mK_X} : X \dasharrow B \subset |mK_X|^{*}$$
	be the pluricanonical map of $X$. 
	Then $f$ descends to an automorphism of $B$ of finite order. 
	In particular, if the Kodaira dimension of $X$ is positive, i.e., if $\kappa (X) \ge 1$, then $\sZ (f) = \emptyset$.
\end{theorem}

\begin{remark}\label{remfield} Theorem~\ref{thmNZ} and some other results that we will cite are formulated and proved over the complex number field
	$\C$ and not over $\oQ$. 
	However, all results over $\C$ which we will cite are valid also over $\oQ$. 
	This is because either the
	proof works also over $\oQ$ with obvious modifications, or
	one can reduce to the statement over $\C$ 
	by taking the base change under $\oQ \subset \C$.

For instance, one can deduce Theorem~\ref{thmNZ} over $\oQ$ from Theorem~\ref{thmNZ} over $\C$ by the second method, as the pluricanonical maps commute with field extensions and $\Aut(W) \subset \Aut(W_{\C})$ for a variety $W$ defined over ${\oQ}$.
\end{remark}

\ssec{Results about the Kawaguchi--Silverman Conjecture}\label{ssec-KSC}
\hfill

In this paragraph, all varieties and maps are defined over $\Qbar$.

\begin{lemma}\label{lem20} Let $X$ be a smooth projective variety and $f \in \Rat (X)$. If $\delta_f = 1$, then KSC (1) and (2) hold for $(X,f)$.
\end{lemma}

\begin{proof}
	This is well-known and easy. Indeed, by~\eqref{ineq-dalimsup},  we have $1 = \delta_f \ge \overline{a}_f(x) \ge \underline{a}_f(x) \ge 1$
	for all $x \in X_f(\oQ)$. Thus $a_f(x) = 1 = \delta_f$ whenever $x \in X_f(\oQ)$. This implies the result.
\end{proof}

\begin{lem}\label{lem-KSCit}
	 Let $f \in \End(X)$. Then KSC (2) holds for $(X,f)$ if and only if KSC (2) holds for $(X,f^N)$ for some $N \in \Z_{>0}$.
\end{lem}

\begin{proof}

See~\cite[Lemma 3.3]{Sa20}.
\end{proof}

\begin{lemma}\label{lem23} Let $X = X_d$ be a smooth projective variety of dimension $d$ and $f = f_d \in \Bir (X)$. Assume that there are smooth projective varieties $X_m$ of dimension $m$ ($1 \le m \le d$), dominant rational maps $\varphi_m : X_m \dasharrow X_{m-1}$ ($2 \le m \le d$) and $f_m \in \Bir (X_m)$ such that $\varphi_m \circ f_m = f_{m-1} \circ \varphi_m$ for all $2 \le m \le d$.
Then $\delta_f = 1$ and KSC (1) and (2) hold for $(X,f)$.
\end{lemma}

\begin{proof} Since $\dim X_m - \dim X_{m-1} = 1$, by Lemma~\ref{lem32},
we have $\delta_{f_{m}} = \delta_{f_{m-1}}$
for all $2 \le m \le d$. Moreover $\delta_{f_{1}} = 1$ as $X_1$ is a smooth projective curve and $f_1 \in \Bir (X_1) = \Aut (X_1)$. Hence $\delta_{f_{m}} = 1$ for all $1 \le m \le d$. Thus $\delta_f = 1$ and the result follows from Lemma~\ref{lem20}.
\end{proof}

The following lemma is similar to~\cite[Theorem 3.4 (i)]{MSS18}.

\begin{lem}\label{lem-liminf}

	Let $\nu : X \dto Y$ be a birational map between smooth projective varieties.
	Let $f \in \Rat (X)$ and $g \in \End (Y)$ such that $g \circ \nu = \nu \circ f $.
	Let $x \in \cZ(f)$.
	Then
	$$\ol{a}_f(x) \ge {a}_g(\nu(x)) \ \ \text{ and } \ \ \ul{a}_f(x) \le {a}_g(\nu(x)).$$
	If moreover $f^k(x) \not\in I(\nu)$ for all but finitely many $k \in \Z_{\ge 0}$
	(e.g. when $\nu$ is a morphism),
	then
	$$\ul{a}_f(x) = {a}_g(\nu(x)).$$
\end{lem}
\begin{proof}
As $\nu$ is a birational map, we may choose a Zariski dense open subset $U$ of $X$
such that $\nu|_{U} : U \to \nu(U)$ is an isomorphism.
Since $\obfx$ is Zariski dense in $X$,
there exists a strictly increasing sequence
$\{n_k\}_{k \ge 1}$ of positive integers
such that $\{f^{n_k}(x)\}_{k \ge 1} \subset U$.

By~\cite[Lemma 3.3]{MSS18},
for an ample Cartier divisor $H_X$ (resp. $H_Y$) on $X$ (resp. $Y$), there are positive constants $M_1$, $L_1$ and constants $M_2$ and $L_2$ such that
$$M_1h_{H_Y}(\nu(x)) +M_2 \le h_{H_X}(x) \le L_1h_{H_Y}(\nu(x)) + L_2$$
for all $x \in U$.
In particular, this implies that
\begin{equation}
	\begin{split}
		\ol{a}_f(x) & \ge \limsup_{k \to \infty} h_{H_X}(f^{n_k}(x))^{\frac{1}{n_k}} \\
		& = \limsup_{k \to \infty} h_{H_Y}(g^{n_k}(\nu(x)))^{\frac{1}{n_k}}
		= {a}_g(\nu(x)).
	\end{split}
\end{equation}
Here for the last equality, we used the fact that KSC (1) holds
for any endomorphism of a
(normal) projective variety~\cite[Theorem 3]{KS16a}.
Similarly, we have $\ul{a}_f(x) \le {a}_g(\nu(x))$.

Now we prove the second statement.
Since
$$ \gS \cnec \Set{ k \in \Z_{\ge 0} | f^k(x) \in I(\nu) }$$
is finite by assumption, we have
$$\ul{a}_f(x) =  \liminf_{k \to \infty} h_{H_X}(f^{k}(x))^{\frac{1}{k}} = \liminf_{k \to \infty, k \notin \gS} h_{H_X}(f^{k}(x))^{\frac{1}{k}},$$
$$\ul{a}_g(\nu(x)) =  \liminf_{k \to \infty} h_{H_Y}(g^{k}(\nu(x)))^{\frac{1}{k}}   = \liminf_{k \to \infty, k \notin \gS}  h_{H_Y}(g^{k}(\nu(x)))^{\frac{1}{k}}.$$
By~\cite[Lemma 3.3]{MSS18}, there exist constants $M$, $M'$ with $M > 0$ such that for any $k \notin \gS$,
$$h_{H_X}(f^k(x)) \ge Mh_{H_Y}(\nu(f^k(x))) + M'.$$
It follows that
$$\ul{a}_f(x) =  \liminf_{k \to \infty, k \notin \gS} h_{H_X}(f^{k}(x))^{\frac{1}{k}} \ge
 \liminf_{k \to \infty, k \notin \gS} h_{H_Y}(g^{k}(\nu(x)))^{\frac{1}{k}}   = \ul{a}_g(\nu(x)) = {a}_g(\nu(x)),$$
 where the last equality follows from~\cite[Theorem 3]{KS16a} as $g \in \End (Y)$.
Hence 	$\ul{a}_f(x) = {a}_g(\nu(x))$.
\end{proof}

\begin{lemma}\label{lem21}
Let $X$ be a smooth projective variety and $f \in \Rat (X)$.
Let $\nu : X \dto Y$ be a birational map to a smooth projective variety $Y$.
Assume that there is $g \in \End (Y)$ such that $g \circ \nu = \nu \circ f$.
Finally, assume that for every  $x \in \cZ(f)$,
we have $f^k(x) \not\in I(\nu)$ for all but finitely many $k \in \Z_{\ge 0}$.
(Note that the last assumption always holds  if $\nu$ is a morphism.)

If KSC (2) holds for $(Y,g)$, then the KSC (2) also holds for $(X,f)$.

\end{lemma}

\begin{proof}

	Let $x \in \cZ(f)$.
	Replacing $x$ by some $f^{k}(x)$,
	we can assume that $x \notin I(\nu)$.	
	Since $g \circ \nu = \nu \circ f$,
	it follows from Lemma~\ref{lem-dorb} that $\nu (x) \in \sZ (g)$.
	Suppose that KSC (2)
	holds for $g$, then $a_{g}(\nu(x)) = \gd_{g}$.
	We thus have the chain of inequalities
$$\gd_f \ge  \ol{a}_f(x) \ge \underline{a}_f(x) = a_{g}(\nu(x)) = \delta_{g} = \delta_{f},$$
where $\gd_f \ge  \ol{a}_f(x)$  follows from~\eqref{ineq-dalimsup},
$\ul{a}_f(x) = a_{g}(\nu(x))$ follows from Lemma~\ref{lem-liminf}, and $\gd_f = \gd_{g}$ follows from~\cite[Corollaire 7]{DS05}.
Thus the limit $a_f(x) = \lim_{n \to \infty} h_H(f^n(x))^{\frac{1}{n}}$
exists and
$\gd_f = a_f(x)$, which completes the proof.
\end{proof}

\begin{lemma}\label{lem22}
	Let $\varphi : X \to B$ be a surjective morphism between smooth projective varieties
	with $\dim X - \dim B \le 1$. Let $f \in \Bir (X)$. Assume that there exists $f_B \in \Aut (B)$ such that $\varphi \circ f = f_B \circ \varphi$.
If KSC (2) holds for $(B,f_B)$, then the KSC (2) also holds for $(X,f)$.
\end{lemma}

\begin{proof}
By the assumption $\dim X - \dim B \le 1$ and by Lemma~\ref{lem32}, it follows that $\delta_f = \delta_{f_B}$. If $x \in \sZ(f)$, then
$\varphi(x) \in \sZ(f_B)$ as well. Then by the assumption, we have $a_{f_B}(\varphi(x)) = \delta_{f_B}$. On the other hand, for an ample divisor $H_B$ on $B$, by the functoriality of the height function (see e.g. \cite[Part B, Theorem B.3.6]{HS00}), we have
$$h_{\varphi^*H_B}(f^{n}(x)) = h_{H_B}(\varphi(f^{n}(x))) + O(1) = h_{H_B}(f_B^{n}(\varphi(x))) + O(1).$$
Combining this with the definition of $\underline{a}_f(x)$, we obtain $\underline{a}_f(x) \ge \underline{a}_{f_B}(\varphi(x)) = \delta_{f_B} = \delta_f$.
By~\eqref{ineq-dalimsup}, we get as before
$$\delta_f \ge \overline{a}_f(x) \ge \underline{a}_f(x) \ge \delta_f.$$
Hence $a_f(x)$ exists and satisfies $a_f(x) = \delta_f$.
\end{proof}


\section{irregular varieties with $\kappa = 0$ and $q = \dim X - 1$} \label{irranydim}

In this section, we study irregular projective varieties with Kodaira dimension $\kappa = \kappa (X) = 0$ and $q=\dim X-1$,
in arbitrary dimension. Our main result in this section is the following.

\begin{pro}\label{pro-NZD}
	Let $X$ be a smooth projective variety over $\C$,
	and let $f \in \Bir(X)$.
	Suppose that $\gk(X) = 0$ and $q(X) = \dim X - 1$.
	Then $f$ does not have any Zariski dense orbit over $\C$.
	
\end{pro}

	Note that Proposition~\ref{pro-NZD} is a result over $\C$,
	and is therefore applicable in both complex dynamics and algebraic dynamics
	(when everything is defined over $\Qbar$ or more generally, over any field of characteristic zero).
	Theorem~\ref{mainthm2} (2) then follows immediately,
	and we will complete the proof of Theorem~\ref{mainthm2}
	at the end of this section.

To prove Proposition~\ref{pro-NZD}, we start with the following two lemmas.

\begin{lem}\label{lem-ell}
		Let $E$ be an elliptic curve over $\C$ and let $f : E \to E$ be an automorphism.
	Then $f$ has infinite order if and only if $f$ is a translation by a non-torsion point.
\end{lem}
\begin{proof}
	Lemma~\ref{lem-ell} is clear if $f$ is a translation.
	Assume that $f$ is not a translation, then $f - \Id_E : E \to E$ is not constant.
	Fix an origin $o$ of $E$.
There exists thus $x_0 \in E$ such that $f(x_0) - x_0 = o$,
	namely $f(x_0) = x_0$.
	If we choose $x_0$ as the new origin of $E$, then $f$ becomes a group automorphism of the elliptic curve $E$,
	hence $f$ has finite order.
\end{proof}

For every projective variety $X$, 
	let $\Aut^0(X)$ denote the identity component of $\Aut (X)$. 
When $A$ is an abelian variety, $\Aut^0(A)$ is the group of translations of $A$.

\begin{lem}\label{lem-prodE}
	Let $A$ be an abelian variety and $E$ an elliptic curve over $\C$.
	Let $g \cnec (g_A,g_E) \in \Aut^0(A) \times \Aut(E)$.
	Let $f \in \Aut(A \times E)$ such that $f\circ g = g \circ f$ and that
	$f$ descends to $f_A \in \Aut(A)$
	through the projection $\pr_1 : A \times E \to A$.
	Assume that $g$ has finite order  and $g_E$ is not a translation on $E$.
	Then $f$ has no Zariski dense orbit over $\C$.
\end{lem}

\begin{proof}
	Since $g_E$ is not a translation,
	it has a fixed point $o \in E$.
	Let $o \in E$ be the origin of $E$, so that $g_E : E \to E$ is a group automorphism.
	
	For every $a \in A$, define
	$$f_a : E = \pr_1^{-1}(a) \xto{f} \pr_1^{-1}(f_A(a)) = E.$$	
	There exist a group automorphism $h\in \Aut_o(E)$ and a morphism $t : A \to E$
	such that
	$$f_a(x) = h(x) + t(a)$$
	for every $x \in E$.
	Here we note that since
	$\Aut_o(E)$ is finite, $h\in \Aut_o(E)$ does not depend on $a \in A$.
	Since $h \circ  g_E  = g_E \circ h$ (because $\Aut_o(E) \subset {\rm GL}(1, \C) = \C^{\times}$ is abelian for a complex elliptic curve $E$; see also \cite[Chapter IV, Corollary 4.7]{Ha77} for a more precise description),
	the condition $f\circ g = g \circ f$ implies
	\begin{equation}\label{eqn-h1h2}
		 t(g_A(a))=  g_E(t(a)).
	\end{equation}
It follows that $t: A \to E$ is constant, because otherwise $t$ is surjective,
and that $g_A$ is a translation together with~\eqref{eqn-h1h2}
implies that $g_E$ is also a translation, which contradicts the assumption.
Thus $\tau \cnec t(a) \in E$ and $f_1 \cnec f_a : E\to E$ are independent of $a \in A$.
	
	Since $g_E$ is not a translation, it has only finitely many fixed points.
	As $g_E$ is a group automorphism,
	$g_E(\ell \cdot \tau) = \ell \cdot \tau$ for every $\ell \in \Z$. So $\tau$ is torsion,
	which implies that $f_1$ has finite order by Lemma~\ref{lem-ell}.
	Thus for every $(x,y) \in A \times E$, we then have
	\begin{equation}\label{incl-YE}
		\Orb_f(x,y) \subset A \times \Set{f_1^{k}(y)| k \in \Z} \subset A \times E,
	\end{equation}
	which shows that $\Orb_f(x,y)$ is not Zariski dense in $Y \times E$ because
	$\Set{f_1^{k}(y)| k \in \Z}$ is finite.
\end{proof}

\begin{proof}[Proof of Proposition~\ref{pro-NZD}]
	
	Let $\ga : X \to A$ be the Albanese map of $X$
	and let $\gamma : A \to A$ be the automorphism induced by $f$,
	namely the one which satisfies $\ga \circ f = \gamma \circ \ga$.
	
	Assume to the contrary that there exists $x \in X$ with Zariski dense $f$-orbit.
	Since $\gk(X) = 0$ and $q(X) = \dim X - 1$,
	by~\cite[Corollary 9.4]{Vi81} there exists an elliptic fiber bundle $X' \to A$
	such that $X$ is birational to $X'$ over $A$.
	Since $X'$ has no rational curves, we have
	a birational morphism $\nu : X \to X'$ over $A$,
	and $f : X\dto X$ descends to an automorphism $f' : X' \to X'$.
	By Lemma~\ref{lem-dorb}, the $f'$-orbit of $y \cnec \nu(x)$ is Zariski dense in $X'$.

	Let $E$ be a fiber of the elliptic fiber bundle $X' \to A$.
	By~\cite[Proposition 1]{Le82},
	there exists a finite \'etale Galois cover $\wt{A} \to A$ of Galois group $G$
	together with a homomorphism $\psi : G \to \Aut(E)$
	such that
	$$(\wt{A} \times E)/G \simeq X'$$
	over $A$; here the $G$-action on $\wt{A} \times E$
	is diagonal, through the covering action on $\wt{A}$
	and the action induced by $\psi$ respectively.
	It follows from~\cite[Lemma 5.2]{DLOZ}
	that up to replacing $f$ (so $f'$) by some finite iteration, which is allowed thanks to Lemma~\ref{lem-itZ}, the automorphism $f' : X' \to X'$ lifts to
	an automorphism $\ti{f} : \wt{A} \times E \to \wt{A} \times E$.
		Consider the group homomorphism
	\begin{equation}
		\begin{split}
			\Z & \to \Aut(G)  \\
			k & \mapsto (g \mapsto \ti{f}^k g \ti{f}^{-k}).
		\end{split}
	\end{equation}
	Since $\Aut(G)$ is finite, there exists $N \in \Z$ such that $\ti{f}^N g \ti{f}^{-N} = g$ for all $g \in G$.
	Therefore up to further replacing $f$ by $f^N$,
	we can assume that $\ti{f}$ commutes with the $G$-action.


	The $G$-action on $E$ does not purely consist of translations.
	Indeed, if $G$ acts on $E$ by translations,
	then $E/G$ is an elliptic curve and
	we have a dominant map $X \to (\wt{A} \times E)/G \to (\wt{A}/G) \times (E/G)$
	onto an abelian variety of dimension equal to $\dim X$,
	which contradicts $q(X) = \dim X - 1$.
	Let $g \in G$ such that $\psi(g) \in \Aut(E)$ is not a translation.
	We use the same notation $g$ to denote the corresponding automorphism in $ \Aut(\wt{A} \times E)$.
	Since $\ti{f} : \wt{A} \times E \to \wt{A} \times E$ commutes with the $G$-action,
	we have $\ti{f} \circ g = g \circ \ti{f}$.
	Thus $\ti{f} \in  \Aut( \wt{A} \times E )$  and $g \in  \Aut( \wt{A}) \times\Aut(E)$
	satisfy all the assumptions in Lemma~\ref{lem-prodE},
	which implies that $\ti{f}$ has no Zariski dense orbit.
	As $\ti{f}$ is a lifting of $f'$ and $f'$ is birational to $f$,
	this contradicts the assumption that
	$\Orb_f(x)$ is Zariski dense.
\end{proof}

\begin{remark}\label{rem-ZODqd-1}
	In the proof of Proposition~\ref{pro-NZD}, one could have also proven that
	Conjecture~\ref{ZDOconj} holds for
	$f \in \Bir(X)$
	where $X$ is a smooth projective variety over $\Qbar$ satisfying $\gk(X) = 0$ and $q(X) = \dim X - 1$.
	Indeed, one can show that
	the birational map $X \dto  (\wt{A} \times E)/G$
	and the lifting $\ti{f} : \wt{A} \times E \cto$ of some positive power of $f$
	are both defined over $\Qbar$.
	Since $\wt{A} \times E \to E$ is $G$-equivariant and $\ti{f}$ commutes with the $G$-action on $\wt{A} \times E$,
	the map $f \in \Bir(X)$ descends to an automorphism of $\P^1$ of finite order through
	the dominant map $X \dto  (\wt{A} \times E)/G \to E/G \simeq  \P^1$ (the last isomorphism because $E/G$ is not an elliptic curve) by \cite[Proposition 2.3]{NZ09} and $\kappa (X) = 0$.
	Conjecture~\ref{ZDOconj} for $f \in \Bir(X)$ then follows.
\end{remark}

\begin{proof}[Proof of Theorem~\ref{mainthm2}]

Let $X$ be a smooth projective variety as in Theorem~\ref{mainthm2}.
Since $\gk(X) = 0$, we have $q(X) \le \dim X$ by~\cite[Theorem 1]{Ka81}. So either $q(X) = \dim X - 1$ or $q(X) = \dim X$.

When $q(X) = \dim X - 1$, Theorem~\ref{mainthm2} is a consequence of
Proposition~\ref{pro-NZD}.
Suppose that $q(X) = \dim X$.
Then the Albanese map $\mu : X \to A$
is birational, and
the rational self-map $f_A$ of $A$ induced by $f : X \dto X$ is a morphism,
i.e.,
$$f_A := \mu \circ f \circ \mu^{-1} \in \End_{{\rm var}} (A).$$
Since KSC (2) holds for any endomorphism
$g \in \End_{{\rm var}} (A)$ of an abelian variety (\cite[Theorem 4]{KS16a}, \cite[Theorem 1.2]{Si17}),
KSC (2) also holds for $f \in \Rat (X)$ by Lemma~\ref{lem21}.
This completes the proof of Theorem~\ref{mainthm2}.
\end{proof}

\section{KSC for irregular threefolds}\label{sect3}

The following result for endomorphisms was proven in~\cite[Proposition 3.7]{LM21}.

\begin{pro}\label{pro-albsurj}
	Let $X$ be a smooth projective variety over $\C$ and let
	$f \in \Rat(X)$.
	If $f$ has a Zariski dense orbit,
	then the Albanese map $\ga : X \to \Al(X)$
	is surjective.
\end{pro}

\begin{proof}
	Suppose that $\ga$ is not surjective.
	By~\cite[Theorems 10.3, 10.9]{Ue75}, the Albanese image
	$B$ is then an \'etale torus fiber bundle (possibly of relative dimension $0$)
	over a projective variety $B'$ of general type with $\dim B' > 0$.
	The rational map $f$ descends to a rational dominant map $f_B$ on $B$.
	Since $B \to B'$ is the Iitaka fibration of $B$,
	the map $f_B$ further descends to a rational dominant map $f_{B'}$ on $B'$.
	As $B'$ is of general type, $f_{B'}$ is of finite order by~\cite[Theorem 1]{KO75}.
	This implies that $f$ has no Zariski dense orbit.
\end{proof}

	We have the following immediate corollary of Proposition~\ref{pro-albsurj}.

\begin{cor}\label{cor-albsurj}
	Let X be a smooth projective variety over $\Qbar$.
	If the Albanese map $\ga : X \to \Al(X)$
	is not surjective, then KSC (2) holds vacuously for any $f \in \Rat(X)$.
\end{cor}

In the rest of this section, we prove Theorem~\ref{mainthm1}.

\medskip
\noindent
{\it Proof of Theorem~\ref{mainthm1}.}
\medskip

Let $X$ be a smooth projective threefold with $q(X) > 0$. By Theorem~\ref{thmNZ},
we can assume that $\kappa (X) = 0$ or $-\infty$. Theorem~\ref{mainthm1} (1) and (2) will follow from
Corollary~\ref{cor-albsurj}, Theorem~\ref{mainthm2}, together with
Propositions~\ref{prop3bis2} and~\ref{prop33} below.

\begin{lem}\label{lem-k0min} Assume that  $\kappa (X) = 0$ and $q(X)>0$.
	Then the minimal model $Y$ of $X$ is smooth with
	$$\Aut(Y) = \Bir (Y).$$
\end{lem}

\begin{proof}  Since $\dim X = 3$ and $\kappa (X) = 0$, it follows that $X$ has a minimal model,
	i.e., a normal projective $\Q$-factorial variety with only terminal (in particular, isolated) singularities.
	Any minimal model $Y$ of $X$ satisfies $\sO_Y(mK_Y) \simeq \sO_Y$
	for some $m$ by the abundance theorem for minimal threefolds (See e.g. \cite{Ka92} and references therein).
	Moreover, the Albanese map $\ga : Y \to  \Al (Y)$
	is a surjective \'etale fiber bundle by~\cite[Theorem 8.3]{Ka85}.
Since $Y$ has only isolated singularities and $ \dim \Al (Y) > 0$,
		 every fiber $F$ of $\ga$ is smooth.  As $K_Y$ is torsion,  it follows that $K_F$ is torsion as well by adjunction.
	In particular,  $Y$ is smooth  and contains no isolated rational curve. Indeed, since $\ga : Y \to  \Al (Y)$ is a surjective \'etale fiber bundle over an elliptic curve $\Al (Y)$ as observed above, every rational curve on $Y$ is in a fiber and it also deforms in the family $\ga : Y \to  \Al (Y)$.
	Therefore $Y$ has no flop, so  the minimal model of $X$ is unique and
	$\Bir (Y) = \Aut (Y)$ by~\cite[Theorem 1]{Ka08}.
\end{proof}

\begin{proposition}\label{prop3bis2}
	Assume that $\kappa (X) = 0$ and $q(X) = 1$.
	Then KSC (2) holds for any $f \in \Bir (X)$.
\end{proposition}

\begin{proof}

	Let $f \in \Bir (X)$.
	Since $\kappa (X) = 0$ and $q(X) = 1$, the Albanese map
	$$\ga : X \to E := \Al (X)$$
	is a surjective morphism onto an elliptic curve
	with connected fibers by~\cite[Theorem 1]{Ka81}
	and $f$ descends to a biregular map $f_E : E \to E$ on $E$.
	In particular, if $x \in \sZ(f)$, then $\ga(x) \in \sZ (f_E)$.
	
	Let $\nu : X \dto Y$ be the birational map to the
	a minimal model $Y$ of $X$; $Y$ is smooth by Lemma~\ref{lem-k0min}.

	\begin{claim}\label{claim-ndom}
		The Albanese image $\ga(I(\nu))$ of the indeterminacy locus $I(\nu) \subset X$ is finite.
	\end{claim}

 We give two proofs of Claim~\ref{claim-ndom}, 
 both of which will be applicable for other problems.

\begin{proof}[First proof of Claim~\ref{claim-ndom}]
		By the universal property of the Albanese map, $\ga$ factors through  the birational map $\nu : X \dto Y$.
		Consider the restriction $\nu_{\bk(E)} : X_{\bk(E)} \dto Y_{\bk(E)}$ of
		$\nu$ where $X_{\bk(E)}$ and $Y_{\bk(E)}$ denote the generic fibers of
		$X$ and $Y$ over $E$ respectively.
		Since $Y$ is a smooth threefold, it follows that 
		$Y_{\bk(E)}$ is a smooth surface.
		As $K_Y$  is $\Q$-linearly trivial, so is $K_{Y_{\bk(E)}}$.
		It follows that $\nu_{\bk(E)}$ is a birational map between smooth $\bk(E)$-surfaces
		with $Y_{\bk(E)}$ being minimal,
		so $\nu_{\bk(E)}$ is a morphism~\cite[Corollary 1 in II.7.3]{IS96}. Hence $I(\nu) \subset X$ does not dominate $E$, which proves Claim~\ref{claim-ndom}.
\end{proof}

\begin{proof}[Second proof of Claim~\ref{claim-ndom}] 
The minimal model program consists of a sequence of divisorial contractions and flips. 
By the universal property of the Albanese variety, 
	these birational modifications are defined over $E$.
Let $X_i \dashrightarrow X_i^+$ be a flip in the program and let $C_i$ be the flipping curve. 
Then one has
$$ \ga(I(\nu)) \subset  \bigcup_i \ga_i(C_i) $$
where $\ga_i : X_i \to E$ is the Albanese map~\cite[Lemma 8.1]{Ka85}.
As each $\ga_i(C_i) \subset E$ is finite (because $C_i$ is a union of finitely many $\P^1$) and there are only finitely many flips in the minimal model program, $\alpha(I(\nu))$ is thus finite. 
This proves Claim~\ref{claim-ndom}.
\end{proof}

For every $x \in \cZ(f)$,
it follows that $f^k(x) \not\in I(\nu)$ for all but finitely many $k \in \Z_{\ge 0}$.
Indeed, if the contrary holds, then, since $\ga(I(\nu))$ is finite,
there exist integers $k_1 < k_2$ such that $\ga(f^{k_1}(x)) = \ga(f^{k_2}(x)) \in \ga(I(\nu))$.
It would follow that $\ga(f^{k_1+m}(x))=\ga(f^{k_2+m}(x))$ for all $m \ge 0$ and hence for any $k \ge k_2$, $\ga(f^{k}(x))=\ga(f^{l}(x))$ for some $\ell \le k_2$. Thus
$$\Orb_f(f^k(x)) \subset \bigcup_{\ell = 1}^{k_2} \ga^{-1}(\ga(f^\ell(x))),$$
which contradicts the Zariski density of $\obfx$.

Hence we can apply Lemma~\ref{lem21} and reduce Proposition~\ref{prop3bis2} to
KSC (2) for $(Y, f_Y)$ with $f_Y \cnec \nu \circ f \circ \nu^{-1}$.
The latter follows from~\cite[Corollary 1.5]{LS21} (or rather its proof):
indeed, $f_Y \in \Aut (Y)$ by Lemma~\ref{lem-k0min}
and the Bogomolov decomposition of an \'etale cover of $Y$ has no Calabi-Yau threefold as its irreducible factor so that we may argue as in the proof of~\cite[Corollary 1.5]{LS21}.
\end{proof}

\begin{proposition}\label{prop33} Assume that $\kappa (X) = -\infty$ and $q(X)>0$. Then
\begin{enumerate}
\item KSC (2) holds for any $f \in \Aut (X)$.
\item Moreover, KSC (2) holds for any $f \in \Bir (X)$ unless $X$ is covered by rational surfaces.
\end{enumerate}
\end{proposition}

\begin{proof} By Corollary~\ref{cor-albsurj}, we can assume that	
the Albanese map
$$\ga : X \to B  := \Al (X)$$
is surjective.

First we prove (2).
Since $\kappa (X) = -\infty$ and $q(X) > 0$, we have $\dim B = 1$ or $2$.
Suppose first that $B$ is an abelian surface. Note that any $f \in \Bir (X)$ descends to a $f_B \in \Aut(B)$,
and KSC (2) holds for automorphisms on smooth surface $B = \Al(X)$ by~\cite[Theorem 2 (c)]{KS14}.
Applying Lemma~\ref{lem22} to the surjective morphism $\ga : X \to B$ yields KSC (2) for $X$.

Now we treat the case where $B$ is an elliptic curve.
Since $\kappa(X) = -\infty$, the variety $X$ is covered by rational curves
(as the consequence of the minimal model program and
the abundance theorem for threefolds,
see e.g. \cite{Ka92} and references therein).
Since $B$ is an elliptic curve, a general fiber $X_t$ of $\ga : X \to B$ is thus uniruled, i.e., covered by rational curves. As poined out by the referee, one can also conclude the uniruledness of $X_t$ by using the addition formula \cite{Ka82}.
Note that $q(X_t)$ is constant in $t$ for smooth fibers $X_t$.

Suppose first that $q(X_t) = 0$ for a general fiber $X_t$. Then
$X_t$ is a rational surface by the classification of surfaces.
In particular, $X$ is covered by rational surfaces,
which is excluded from the assertion (2).

Thus we may assume that  $q(X_t) = q >0$ for a general fiber $X_t$. Then $q(X_{\eta}) = q > 0$ as well for the generic fiber $X_{\eta}$. Let $S_{\eta}$ be the image of the Albanese map of $X_{\eta}$ over $\eta = {\rm Spec}\, \oQ(B)$.
Then by taking a model of $X_{\eta} \to S_{\eta} \to {\rm Spec}\, \oQ(B)$ over $B$, we have dominant rational maps
$$X \dasharrow S \dasharrow B$$
which are $\Bir (X)$-equivariant with $X$ and $S$ being smooth.
Thus $\delta_f = 1$ and KSC (2) holds for all $f \in \Bir (X)$ by Lemma~\ref{lem23}.

It remains to show that KSC (2) holds for $f \in \Aut (X)$ for the case excluded in (2), namely when $X$ is covered by rational surfaces.
If $\delta_f = 1$,  then KSC (2) holds by Lemma~\ref{lem20}. So we may and will assume that $\delta_f > 1$.
Moreover, we may and will assume that the induced automorphism $f_B \in \Aut (B)$ is of infinite order, otherwise $f_B$, and thus $f: X \to X$, have no Zariski dense orbit. 
Hence $f_B$ is a translation automorphism $t_s : x \mapsto x+s$ of infinite order by Lemma~\ref{lem-ell}.
Since $\ga : X \to B$ has at most finitely many singular fibers and
$\ga \circ f = f_B \circ \ga$, necessarily
$\ga : X \to B$ is a smooth morphism.

Now we apply the classification result of
Lesieutre~\cite[Theorem 1.7]{Le18} to describe $f \in \Aut(X)$.
First of all, if $f \in \Aut(X)$ is not dynamically minimal (see~\cite[Definition 1.2]{Le18} for the definition),
then there exists a divisorial contraction $X \to Y$ such that $f$ descends to an automorphism  $f_Y \in \Aut(Y)$.
Since $\ga_Y \circ f_Y = f_B \circ \ga_Y$ where
$\ga_Y : Y \to B$ is the Albanese map,
the same argument as above shows that $Y$ is smooth.
Repeating the same procedure until it stops (as the Picard number strictly decreases after each divisorial contraction),
we obtain a birational morphism $X \to X'$ to a smooth projective variety
such that $f \in \Aut(X)$ descends to $f' \in \Aut(X')$ and that  $f'$ is dynamically minimal.

As $f' \in \Aut(X')$ is dynamically minimal and $\gk(X) = -\infty$,
by~\cite[Theorem 1.7]{Le18} $X'$, and thus $X$, must have either a conic bundle structure $\pi : X \to V$ or a dominant rational map $\pi : X \dasharrow S$ to a surface $S$. Moreover, there exists $n \in  \Z_{>0}$ such that $f^n \in \Aut(X)$ descends to a birational self-map on $V$ in the first case,
and $f^n \in \Aut(X)$ descends to a (biregular) automorphism on $S$ in the second case.
We may and will assume that $S$ is a smooth projective surface by taking an equivariant resolution of singularities.

If $X$ has a conic bundle $\pi : X \to V$, then the Albanese map $\ga : X \to B$ factors through $V$, because $B$ is an elliptic curve.
Thus $\delta_{f^n} = 1$ by Lemma~\ref{lem23}.
Hence KSC (2) holds for $f^n$ by Lemma~\ref{lem20} and thus for $f$ by Lemma~\ref{lem-KSCit}.

Assume that $X$ has a dominant rational map $\pi : X \dasharrow S$ to a smooth projective surface $S$
such that $f^n$ descends to an automorphism $f_S \in \Aut(S)$.
Let $Z$ be the graph of $\pi$ in $X \times S$. Then we have an automorphism $f_Z$ of $Z$ induced by $f^n \times f_S \in \Aut (X \times S)$. Let $W$ be  an equivariant projective resolution  of singularities of $Z$. Then there is $F_W \in \Aut (W)$ such that the induced birational morphism $W \to X$ is equivariant with respect to $F_W \in \Aut (W)$ and $f^n \in \Aut (X)$, and also the induced morphism $W \to S$ is equivariant with respect to $F_W$ and $f_S$.

Then KSC (2) holds for $F_W$ by Lemma~\ref{lem22}, as it holds for automorphisms of smooth projective surfaces by~\cite[Theorem 2 (c)]{KS14}.
So KSC (2) holds for $f^n$ by~\cite[Corollary 3.2]{LS21}, and thus for $f$ by Lemma~\ref{lem-KSCit}.
This completes the proof of the assertion (2).
\end{proof}

As mentioned, Theorem~\ref{mainthm1} (1) and (2)  follow from
Corollary~\ref{cor-albsurj}, Theorem~\ref{mainthm2}, together with
Propositions~\ref{prop3bis2} and~\ref{prop33}.
\qed

\section{Zariski closure of orbits}\label{sect4}

 In this section we prove Theorem~\ref{mainthm3}.

\begin{proof}[Proof of Theorem~\ref{mainthm3} (1)]
	By Theorem~\ref{thmNZ}, $\gk (X) \in \{0, -\infty\}$ if $\sZ(f) \not= \emptyset$.
If $\sZ(f) \not= \emptyset$, then the Albanese map $\ga : X \to \Al (X)$ is surjective by Corollary~\ref{cor-albsurj}. In particular, $q(X) \le d$.
If $\kappa (X) = 0$, then $q(X) \not= d-1$ by Proposition~\ref{pro-NZD}. If $\kappa (X) = -\infty$, then the surjective Albanese map can not be generically finite, 
and hence $q(X) \le d-1$.
This completes the proof of Theorem~\ref{mainthm3} (1).
\end{proof}

It remains to show Theorem~\ref{mainthm3} (2).
Let $E$ be an elliptic curve over $\Qbar$.

\begin{lem}\label{lem-densEn}
	Let $f \in \SL(n,\Z)$. Assume that the characteristic polynomial $\Phi(f)$ of $f$ satisfies the property that if
	$x_1, \ldots,x_n$ are its roots, then $x_i\ol{x_j} \ne 1$ for all $i$ and $j$.
	Then the induced automorphism $f : E^n \to E^n$ has a Zariski dense $\Qbar$-orbit.
\end{lem}

\begin{proof}
	Assume to the contrary that $f : E^n \to E^n$ has no Zariski dense $\Qbar$-orbits.
	Then by~\cite[Theorem 1.2]{GS17}, there exists a dominant rational map $p : E^n \dto \P^1$ such that $p \circ f = p$.
	Such $p$ is defined by a linear system $\Lambda$ on $E^n$ whose base locus has codimension at least $2$.
	Since $p \circ f = p$, we have $f^*D = D$ for every $D \in \Lambda$, which implies that the action $f^* : H^{1,1}(E_\C^n) \cto$ has an eigenvalue $\gl = 1$.
	Since $f^* : H^{1,0}(E_\C^n) \cto$ coincides with the automorphism $f : \C^n \cto$ defined by $f \in \SL(n,\Z)$,
	the eigenvalues of
$$f^* : H^{1,1}(E_\C^n) \simeq H^{1,0}(E_\C^n) \otimes \ol{H^{1,0}(E_\C^n)} \cto$$
are $x_i\ol{x_j} $ with $i,j \in \{1,\ldots,n\}$. Thus $x_i\ol{x_j} = 1$ for some $i$ and $j$.
\end{proof}

We will need the notion of {\it Pisot unit} in our construction.
We call a real algebraic integer $a >1$ a Pisot number if its Galois conjugates except $a$ are in the open unit disk $\Set{z \in \C | {|z|} <1}$. A Pisot unit is a Pisot number which is invertible in the ring of algebraic integers. See e.g. \cite[Section 3]{Og19} for summary what we need and \cite{BDGPS92} for details.

Let $m$ be an integer such that $m \ge 2$ and let $M_{m} \in \SL(m, \Z)$
such that the characteristic polynomial $\Phi_{M_m}(T)$ of $M_m$ is
the minimal polynomial of a Pisot number $a_m >1$ of degree $m$.
Such $a_m$ and $M_m$ exist;
later we may and will even assume that $a_m$ is a Pisot unit in $\Q(\sqrt[m]{2})$ or $\Q(\sqrt[m]{3})$,
see~\cite[Theorem 3.3]{Og19}.
Let $f_m \in \Aut (E^m)$ be the automorphism defined by $M_{m}$.

\begin{pro}\label{0d}
	We have $\delta_{f_m} = a_m^2 >1$ and $\cZ(f_m) \ne \emptyset$.
\end{pro}

\begin{proof} For the first statement, see~\cite[Theorem 4.2]{Og19}.
	Assume that $m = 2$. Then since $\delta_{f_m} >1$, we have $\cZ(f_m) \ne \emptyset$ 
	by a result of Xie~\cite[Theorem 1.3]{Xi15} cited in Introduction.
	Now assume that $m \ge 3$.
	Let $x_1,\ldots,x_{m-1},a_m$ be the roots of $\Phi_{M_m}(T)$.
	By construction we have $a_m^2 > 1$ and $|x_ix_j| <1$ for all $i$ and $j$.
	If $a_m\ol{x_i} = 1$ for some $i$ (e.g. $i = m-1$) then
	$$x_1 \cdots x_{m-2}= \Phi_{M_m}(0) \in \Z$$
	which is impossible because $0<|x_i| < 1$.
\end{proof}

Let $n \ge 0$ be an integer. If $n \ge 2$, we define $h_n \in \Aut (E^n)$ by a matrix
$N_{n} \in {\rm SL}(n, \Z)$ such that the characteristic polynomial $\Phi_{N_n}(T)$ of $N_n$ is the minimal polynomial of a Pisot unit $b_n >1$ of degree $n$ in $\Q(\sqrt[n]{3})$.
As before, such $b_n$, $N_n$ and $h_n$ exist and $\delta_{h_n} = b_n^2 >1$.
If $n=1$, we choose $h_1$ to be a translation of $E$ of infinite order.
We put $\Phi_{N_1}(T) := T-1$ and $b_1 \cnec 1$. Note that $\delta_{h_1} = 1$.
If $n=0$, we set $h_0$ to be the identity of the point,
	$\Phi_{N_0}(T) \cnec 1$ and $b_0 \cnec 1$.

From now on, we assume that $a_m$ is a Pisot unit in $\Q(\sqrt[m]{2})$.
Set
$$B_{m,n} := E^m \times E^n\,\, ,\,\, \phi_{m,n} := f_m \times h_n \in \Aut(B_{m,n}).$$
By $m \ge 2$, we have
$$\delta_{\phi_{m,n}} = \max(\delta_{f_m},\delta_{h_n}) = \max (a_m^2, b_n^2) > 1.$$

\begin{lem}\label{c_ab}
We have $\cZ(\phi_{m,n}) \ne \emptyset$.
\end{lem}

\begin{proof}

	Let $x \in E^m$ and $y \in E^n$
	such that $\Orb_{f_m}(x)$ and $\Orb_{h_n}(y)$ are Zariski dense in $E^m$ and $E^n$ respectively.
	Set $B := B_{m, n}$ and $\phi := f_m \times h_n$.
	Let $W \subset B$ be any irreducible component of the Zariski closure of $\Orb_{f_m \times h_n}((x,y))$.
	Up to replacing $\phi$ by some positive power of it,
	we can assume that $\phi (W) = W$.
	In particular, for the same reason as in the proof of Proposition~\ref{pro-albsurj}, $\sZ(\phi|_W) \ne \emptyset$ under this replacement and therefore $W$ is a translation of an abelian subvariety of the abelian variety $B = E^m \times E^n$.
	Note that the projections ${\rm pr}_1 : W \to E^m$ and ${\rm pr}_2 : W \to E^n$ are surjective, as both are equivariant under $\phi$, $f_m$ and $h_n$ and $\sZ(\phi|_W) \not= \emptyset$.
	Hence
	$${\rm pr}_1^{*} : H^{1,0}(E^m_{\C}) \to H^{1,0}(W_{\C})\,\, ,\,\, {\rm pr}_2^{*} : H^{1,0}(E^n_{\C}) \to H^{1,0}(W_{\C})$$
	are both injective and equivariant under $\phi$, $f_m$ and $h_n$.
	
	Note that the characteristic polynomial of $f_m^*|H^{1,0}(E^m_{\C})$ (resp. $h_n^*|H^{1,0}(E^n_{\C})$) is $\Phi_{M_m}(T)$ (resp. $\Phi_{N_n}(T)$). 
	Thus both $\Phi_{M_m}(T)$ and $\Phi_{N_n}(T)$ divide the characteristic polynomial $\Phi(T)$ of $\phi^*|H^{1,0}(W_{\C})$. Since $\Phi_{M_m}(T)$ and $\Phi_{N_n}(T)$ are 
	the minimal polynomials of $a_m$ and $b_n$ respectively, 
	and different Pisot numbers are not conjugate to each other,
	$\Phi_{M_m}(T)$ and $\Phi_{N_n}(T)$ are coprime in $\C[T]$.
	It follows that $\Phi(T)$ is divided by $\Phi_{M_m}(T)\Phi_{M_n}(T)$. 
	In particular, 
	$$\dim H^{1,0}(W_{\C}) \ge m+n = \dim H^{1,0}(B_{\C}).$$ 
	Since $W$ is a translation of an abelian subvariety of an abelian variety $B$, 
	it follows that $W = B$. This completes the proof.
 \end{proof}

\begin{proposition}\label{00}
There exist a  variety $Y$ with $(\kappa,q)=(0,0)$  of dimension $d \ge 2$ and $g \in \Aut(Y)$ with $\gd_g > 1$ and $\sZ(g) \ne \emptyset$.
\end{proposition}

\begin{proof}
Let $\wt{E^m}$ be the blow up of $E^m$ at the subset of two torsion points $E^m[2]$ of $E^m$. 
Let $\iota_m$ be the automorphism of $\wt{E^m}$ induced by $-\Id_{E^m}$. 
Let $Y_m := \wt{E^m}/\iota_m$. Then $Y_m$ is a smooth projective variety with
$$(\dim Y_m, \kappa(Y_m), q(Y_m)) = (m, 0, 0).$$
$f_m$ naturally induces a biregular automorphism of $Y_m$, which we denote by $g_m$. We have then $\delta_{g_m} = \delta_{f_m} >1$ by Lemma~\ref{lem32}.
Finally, if we choose $x \in \sZ(f_m) \setminus E^m[2]$,
then its image $y \in Y_m$ belongs to $\sZ(g_m)$.
\end{proof}

\begin{proposition}\label{c_cyab} Assume that $d \ge 3$.
For any  $1 \le n \le d-2$, there exist a $d$-dimensional variety
  $Z_{n}$  with  $(\kappa, q) = (0, n)$ and $\varphi_{d-n,n} \in \Aut(Z_{n})$ with $\gd_{\varphi_{d-n,n} } > 1$ such that $\sZ (\varphi_{d-n,n}) \not= \emptyset$.
\end{proposition}

\begin{proof}
We set $m:=d-n \ge 2$ and
$$Z_{n}: = Y_{m} \times E^n\,\, ,\,\, \varphi_{m,n} := g_{m} \times h_n \in \Aut(Z_{n}).$$
We have
$$\delta_{\varphi_{m,n}} = \max (\delta_{g_{m}}, \delta_{h_n}) = \max (\delta_{f_{m}}, \delta_{h_n}) > 1,$$
$$(\dim Z_{n}, \kappa(Z_{n}), q(Z_{n})) = (d, 0, n).$$

 Choose $x \in E^m$ and $y \in Y_m$ as in the proof of Proposition~\ref{00}
 and choose $v \in E^n$ such that ${\rm Orb}_{h_n}(v)$ is Zariski dense in $E^n$.
Then, $(x, v) \in \sZ(\phi_{m, n})$ on $E^m \times E^n$. 
We deduce from this that $(y, v) \in \sZ(\varphi_{m, n})$ on $Y_m \times E^n$. 
\end{proof}

It remains to prove Theorem~\ref{mainthm3} (2) when $\kappa (X) = -\infty$.

\begin{proposition}\label{cyproj} Let $d \ge 2$.
For any $0 \le n \le d-1$, there exist a $d$-dimensional variety $V_{n}$  with $(\kappa,q)=(-\infty,n)$
and $\phi \in \Aut(V_{n})$ such that $\sZ(\phi) \ne \emptyset$. 
If $(d, \kappa, q) \ne (2,-\infty,1)$, 
	there exists such $\phi$ with $\gd_\phi > 1$.
\end{proposition}

\begin{proof}

Let $\tau$ be an automorphism of $\P^1$ defined by $\tau(t) = t+1$. Here $t$ is the affine coordinate of $\P^1 \setminus \{\infty\}$.

\begin{lemma}\label{c_cyproj} Let $V$ be a smooth projective variety with $\kappa (V) =0$ and let $\rho$ be an automorphism of $V$ with $\sZ(\rho) \not= \emptyset$. Then $\sZ(\tau \times \rho) \not= \emptyset$ on $\P^1 \times V$.
Moreover,
$\delta_{\tau \times \rho} >1$ if $\delta_{\rho} >1$. 
\end{lemma}

\begin{proof} Set $\phi := \tau \times \rho$. Let $p \in \P^1 \setminus \{\infty\}$ and $v \in \sZ(\rho) \subset V$.

We are going to show that $(p, v) \in \sZ(\tau \times \rho)$. As before, consider any irreducible component $W$ of the Zariski closure of ${\rm Orb}_{\phi}((p,v))$. Assume to the contrary that $W \not= \P^1 \times V$. As before we may assume that $\phi(W) = W$. Let $\nu : \tilde{W} \to W$ be an equivariant projective resolution of $W$. We denote by $\tilde{\phi} \in \Aut (\tilde{W})$ the lift of $\phi$ on $\tilde{W}$. 
For the same reason as  before,
the projections
$$p_1 := {\rm pr}_1 \circ \nu : \tilde{W} \to \P^1\,\, ,\,\, p_2 := {\rm pr}_2 \circ \nu : \tilde{W} \to V$$
are surjective as $\sZ(\phi_{|W}) \not= \emptyset$. Thus $\tilde{W} \to V$ is a generically finite morphism by $W \not= \P^1 \times V$. Since $\kappa(W) = 0$, it follows that $\kappa (\tilde{W}) \ge 0$. Since $\tau$ has no periodic point on $\P^1_{\C}$ except $\infty$, it follows that
$$p_2 : \tilde{W}_{\C} \to \P^1_{\C}$$
is a smooth morphism over $\P^1 \setminus \{\infty\}$. As $\kappa(\tilde{W}) \ge 0$, this contradicts \cite[Theorem 0.2]{VZ01}. Thus $(p, v) \in \sZ(\tau \times \rho)$ as claimed. Observe that $\delta_{\tau} = 1$ and
$$\delta_{\tau \times \rho} = \max (\delta_{\tau}, \delta_{\rho}) = \delta_{\rho}.$$
This implies the last assertion.
\end{proof}

Consider 
$$(V_n, \phi)=(\P^1 \times E^{d-1}, \tau \times h_{d-1})\,\, , \,\, (\P^1 \times Y_{d-1}, \tau \times g_{d-1})\,\, {\rm and}\,\, (\P^1 \times Z_{d-1-q, q}, \tau \times \varphi_{d-1-q, q})$$ 
with $d \ge 2$ in the first two cases, 
and $1 \le q \le d-3$ in the last one. 
We then obtain the desired examples for $(\kappa, q) = (-\infty,d-1)$, $(-\infty,0)$ and $(-\infty, q)$ with $1 \le q \le d-3$.

It remains to construct examples for $(\kappa, q) = (-\infty, d-2)$ 
with $d \ge 2$.
We now choose $E$ as the elliptic curve defined by the Weierstrass equation $y^2 = x^3-1$. $E$ has the automorphism $\tau$ of order $6$
defined by
$$\tau : (x, y) \mapsto (\omega x, -y)\,\, {\rm where}\,\, \omega := \frac{-1 + \sqrt{-3}}{2}.$$
Let  $f_2 \in \Aut (E^2)$ be defined as above.
The minimal resolution $W$ of the quotient surface $E^2/\langle (\tau, \tau) \rangle$ is a smooth projective rational surface (by the Castelnouvo-Enriques criterion) with an automorphism $f_W$ induced by $f_2$. We have $\delta_{f_W} = \delta_{f_2} > 1$ again by Lemma~\ref{lem32}. Let $(E^n, h_n)$ ($n \ge 0$) as before.

Then, by Lemma~\ref{c_ab}, there exists $(x, v) \in \sZ(f_2 \times h_{d-2})$ in $E^2 \times E^{d-2}$ such that $x$ is not in the fixed locus of $\langle (\tau, \tau) \rangle$ which is a subset of the set of $6$-torsion points of $A$.
Let $y \in W$ be the image of $x$ under $E^2 \dasharrow W$. Then we deduce that $(y, v) \in \sZ(f_W \times h_{d-2})$ from $(x, v) \in \sZ(f_2 \times h_{d-2})$. Moreover $\delta_{f_W \times h_{d-2}} >1$ as $\delta_{f_W} >1$. Thus $(V_{d-2}, \phi):=(W \times E^{d-2}, f_W \times h_{d-2})$ gives a desired example for $(\kappa, q) = (-\infty, d-2)$.
\end{proof}

Combining Propositions~\ref{0d}, \ref{00}, \ref{c_cyab}, \ref{cyproj},
Theorem~\ref{mainthm3} (2) now follows.

\begin{remark}\label{exclude}
	If $(d, \kappa, q) = (2,-\infty,1)$, then $X$ is a ruled surface over the elliptic curve $\Al(X)$ and thus $\delta_f = 1$ for all $f \in \Bir (X)$ by Lemma~\ref{lem23}.
\end{remark}

\begin{remark}\label{cyrat}
	Let $E$ be the elliptic curve defined by the Weierstrass equation $y^2 = x^3-1$ with order $6$ automorphism $\tau$ as above.
	We define $f_3 \in \Aut (E^3)$ as above.
	Let $Y_3$ be the blow up at the maximal ideal of the singular points of $E^3/ \langle (\tau^2, \tau^2, \tau^2) \rangle$ and $Z_3$ be the blow up at the maximal ideal of the singular points of $E^3/ \langle (\tau, \tau, \tau) \rangle$.
	Then $f_3$ induces automorphisms $f_{Y_3} \in \Aut (Y_3)$ and $f_{Z_3} \in \Aut (Z_3)$.
	For the same reason as above, both $f_{Y_3}$ and $f_{Z_3}$ have first dynamical degree $>1$ and Zariski dense orbit.
	Here $Y_3$ is a simply-connected Calabi-Yau threefold and $Z_3$ is a smooth rational threefold \cite{OT15}.
\end{remark}

\end{document}